\documentclass{amsart}
\usepackage{amsfonts}
\usepackage{mathrsfs}
\usepackage{amsthm}
\usepackage{amssymb}
\usepackage{amsmath}
\usepackage{enumerate}

\theoremstyle{plain}
\newtheorem{theorem}{Theorem}[section]

\newtheorem{corollary}[theorem]{Corollary}
\newtheorem{lemma}[theorem]{Lemma}

\theoremstyle{definition}
\newtheorem{definition}{Definition}[section]

\theoremstyle{remark}
\newtheorem{remark}{Remark}[section]

\theoremstyle{example}
\newtheorem{example}{Example}[section]

\numberwithin{equation}{section}

\begin{document}

\title[Characterization Theorems]
{Characterization Theorems for Generalized Functionals of Discrete-Time Normal Martingale}

\author{Caishi Wang}
\address[Caishi Wang]
          {School of Mathematics and Statistics,
          Northwest Normal University,
          Lanzhou, Gansu 730070,
          People's Republic of China }
\email{wangcs@nwnu.edu.cn}

\author{Jinshu Chen}
\address[Jinshu Chen]
         {School of Mathematics and Statistics,
         Northwest Normal University,
          Lanzhou, Gansu 730070,
          People's Republic of China}


\begin{abstract}
In this paper, we aim at characterizing generalized functionals of discrete-time normal martingales.
Let $M=(M_n)_{n\in \mathbb{N}}$ be a discrete-time normal martingale that has the chaotic representation property.
We first construct testing and generalized functionals of $M$ with an appropriate orthonormal basis for $M$'s square integrable functionals.
Then we introduce a transform, called the Fock transform, for these functionals and characterize them via the transform.
Several characterization theorems are established. Finally we give some applications of these characterization theorems.

Our results show that generalized functionals of discrete-time normal martingales can be characterized only by growth condition,
which contrasts sharply with the case of some continuous-time processes (e.g., Brownian motion), where both growth condition and analyticity
condition are needed to characterize generalized functionals of those continuous-time processes.
\end{abstract}

\maketitle

\section{Introduction}\label{sec-1}

Hida's white noise analysis is essentially an infinite dimensional calculus on generalized functionals
of Brownian motion \cite{hida,huang,kuo,obata}. In 1988, Y. Ito \cite{ito} introduced his theory of generalized Poisson functionals, which can be viewed as
an infinite dimensional calculus on generalized functionals of Poisson martingale.
It is known that both Brownian motion and Poisson martingale are continuous-time normal martingales.
There are theories of white noise analysis for some other continuous-time processes (see, e.g., \cite{albe, barhoumi, di, hu,lee}).

Discrete-time normal martingales \cite{privault} also play an important role in many theoretical and applied fields \cite{mot,rud}.
It would then be interesting to develop an infinite dimensional calculus on generalized functionals
of discrete-time normal martingales. In \cite{wang-z}, the authors defined the Wick product for generalized functionals of Bernoulli noise and
analyzed its properties. In fact, generalized functionals of Bernoulli noise can be viewed as generalized functionals of a random walk.

In this paper, we consider a class of discrete-time normal martingales, namely the ones that have the chaotic representation property,
which include random walks, especially the classical random walk. Our main work is as follows.
Let $M=(M_n)_{n\in \mathbb{N}}$ be a discrete-time normal martingale that has the chaotic representation property.
We first construct testing and generalized functionals of $M$ with an appropriate orthonormal basis for $M$'s square integrable functionals.
Then we introduce a transform, called the Fock transform, for these functionals and characterize them via the transform.
Several characterization theorems are established. Finally we give some applications of these characterization theorems.

Our results show that generalized functionals of discrete-time normal martingales
can be characterized only by growth condition, which contrasts sharply with the case of some continuous-time processes (e.g., Brownian motion),
where both growth condition and analyticity condition are needed to characterize
generalized functionals of those continuous-time processes (see, e.g., \cite{hida,hol,huang,kuo,obata,pott}).

\section{Discrete-time normal martingale}\label{sec-2}

Throughout this paper, $\mathbb{N}$ designates the set of all nonnegative integers and $\Gamma$ the finitie power set of $\mathbb{N}$, namely
\begin{equation}\label{eq-2-1}
\Gamma = \{\,\sigma \mid \text{$\sigma \subset \mathbb{N}$ and $\#(\sigma) < \infty$} \,\},
\end{equation}
where $\#(\sigma)$ means the cardinality of $\sigma$ as a set. It is not hard to check that $\Gamma$ is countable as an infinite set.
Additionally, we assume that $(\Omega, \mathcal{F}, P)$ is a given probability space. We denote by $\mathcal{L}^{2}(\Omega, \mathcal{F}, P)$ the usual Hilbert space of
square integrable complex-valued functions on $(\Omega, \mathcal{F}, P)$
and use $\langle\cdot,\cdot\rangle$ and $\|\cdot\|$ to mean its inner product and norm, respectively.
By convention, $\langle\cdot,\cdot\rangle$ is conjugate-linear in its first argument and linear in its second argument.

\begin{definition}\cite{privault}\label{def-2-1}
A (real-valued) stochastic process $M=(M_n)_{n\in \mathbb{N}}$ on
$(\Omega, \mathcal{F}, P)$ is called a discrete-time normal martingale if it is square integrable and satisfies:
\begin{enumerate}[(i)]
  \item $\mathbb{E}[M_0 | \mathcal{F}_{-1}] = 0$ and $\mathbb{E}[M_n | \mathcal{F}_{n-1}] = M_{n-1}$
 for $n\geq 1$;
  \item $\mathbb{E}[M_0^2 | \mathcal{F}_{-1}] = 1$ and $\mathbb{E}[M_n^2 | \mathcal{F}_{n-1}] = M_{n-1}^2 +1$
for $n\geq 1$,
\end{enumerate}
where $\mathcal{F}_{-1}=\{\emptyset, \Omega\}$ and
$\mathcal{F}_n = \sigma(M_k; 0\leq k \leq n)$ for $n \in \mathbb{N}$.
\end{definition}

Now let $M=(M_n)_{n\in \mathbb{N}}$ be a discrete-time normal martingale on $(\Omega, \mathcal{F}, P)$.
We give some necessary notions concerning $M$. First we construct from $M$ a process $Z=(Z_n)_{n\in \mathbb{N}}$ as
\begin{equation}\label{eq-2-2}
Z_0=M_0,\quad Z_n = M_n-M_{n-1},\quad  n\geq 1.
\end{equation}
It can be verified that $Z$ admits the following properties:
\begin{equation}\label{eq-2-3}
    \mathbb{E}[Z_n | \mathcal{F}_{n-1}] =0\quad \text{and}\quad  \mathbb{E}[Z_n^2 | \mathcal{F}_{n-1}] =1,\quad  n\in \mathbb{N}.
\end{equation}
Thus, it can be viewed as a discrete-time noise (see \cite{privault}).

\begin{definition}\label{def-2-2}
The process $Z$ defined by (\ref{eq-2-2}) is called the discrete-time normal noise associated
with $M$.
\end{definition}

The next lemma shows
that, from the discrete-time normal noise $Z$, one can get an orthonormal system in $\mathcal{L}^2(\Omega, \mathcal{F}, P)$,
which is indexed by $\sigma \in \Gamma$.

\begin{lemma}\cite{emery,wang-lc}\label{lem-2-1}
Let $Z=(Z_n)_{n\in \mathbb{N}}$ be the discrete-time normal noise associated with $M$.
Define $Z_{\emptyset}=1$, where $\emptyset$ denotes the empty set, and
\begin{equation}\label{eq-2-4}
    Z_{\sigma} = \prod_{i\in \sigma}Z_i,\quad \text{$\sigma \in \Gamma$, $\sigma \neq \emptyset$}.
\end{equation}
Then $\{Z_{\sigma}\mid \sigma \in \Gamma\}$ forms a countable orthonormal system in
$\mathcal{L}^2(\Omega, \mathcal{F}, P)$.
\end{lemma}

Let $\mathcal{F}_{\infty}=\sigma(M_n; n\in \mathbb{N})$, the $\sigma$-field over $\Omega$ generated by $M$.
In the literature, $\mathcal{F}_{\infty}$-measurable functions on $\Omega$ are also known as functionals of $M$.
Thus elements of $\mathcal{L}^2(\Omega, \mathcal{F}_{\infty}, P)$ can be called square integrable functionals of $M$.

\begin{definition}\label{def-2-3}
The discrete-time normal martingale $M$ is said to have the chaotic representation property if
the system $\{Z_{\sigma}\mid \sigma \in \Gamma\}$ defined by (\ref{eq-2-4}) is total in $\mathcal{L}^2(\Omega, \mathcal{F}_{\infty}, P)$.
\end{definition}

So, if the discrete-time normal martingale $M$ has the chaotic representation property, then the system $\{Z_{\sigma}\mid \sigma \in \Gamma\}$
defined by (\ref{eq-2-4}) is actually an orthonormal basis for $\mathcal{L}^2(\Omega, \mathcal{F}_{\infty}, P)$,
which is a closed subspace of $\mathcal{L}^2(\Omega, \mathcal{F}, P)$ as is known.

\begin{remark}
\'{E}mery \cite{emery} called a $\mathbb{Z}$-indexed process $X=(X_n)_{n \in \mathbb{Z}}$ 
a novation, provided it satisfies (\ref{eq-2-3}), and introduced the notion of the chaotic representation property for such a process.
\end{remark}

\section{Generalized functionals of discrete-time nornal martingale}\label{sec-3}

In the present section, we show how to construct generalized functionals of a discrete-time normal martingale.

Let $M=(M_n)_{n\in \mathbb{N}}$ be a discrete-time normal martingale on $(\Omega, \mathcal{F}, P)$
that has the chaotic representation property.
We denote by $Z=(Z_n)_{n\in \mathbb{N}}$ the discrete-time normal noise associated
with $M$ (see (\ref{eq-2-2}) for its definition) and use the notation $Z_{\sigma}$ as defined in (\ref{eq-2-4}).

For brevity, we use $\mathcal{L}^{2}(M)$ to mean the space of square integrable functionals of $M$, namely
\begin{equation}\label{eq-3-1}
  \mathcal{L}^{2}(M) = \mathcal{L}^{2}(\Omega, \mathcal{F}_{\infty}, P),
\end{equation}
which shares the same inner product and norm with $\mathcal{L}^{2}(\Omega, \mathcal{F}, P)$, namely $\langle\cdot,\cdot\rangle$ and $\|\cdot\|$.

\begin{lemma}\cite{wang-z}\label{lem-wz}
Let $\sigma\mapsto\lambda_{\sigma}$ be the $\mathbb{N}$-valued function on $\Gamma$ given by
\begin{equation}\label{eq-3-2}
\lambda_{\sigma}=
\left\{
  \begin{array}{ll}
    \prod_{k\in\sigma}(k+1), & \hbox{$\sigma\neq \emptyset$, $\sigma\in\Gamma$;}\\
    1, & \hbox{$\sigma=\emptyset$, $\sigma\in\Gamma$.}
  \end{array}
\right.
\end{equation}
Then, for $p>1$, the positive term series $\sum_{\sigma\in\Gamma}\lambda^{-p}_{\sigma}$ converges and moreover
\begin{equation}\label{eq-3-3}
\sum_{\sigma\in\Gamma}\lambda^{-p}_{\sigma}\leq \exp\bigg[\sum_{k=1}^{\infty}k^{-p}\bigg]<\infty.
\end{equation}
\end{lemma}

Using the $\mathbb{N}$-valued function defined by (\ref{eq-3-2}), we can construct a chain of Hilbert spaces of functionals of $M$ as follows.
For $p\geq 0$, we define a norm $\|\cdot\|_p$ on $\mathcal{L}^{2}(M)$ through
\begin{equation}\label{eq-3-4}
  \|\xi\|_{p}^2=\sum_{\sigma\in \Gamma}\lambda_{\sigma}^{2p}|\langle Z_{\sigma}, \xi\rangle|^{2},\quad \xi \in \mathcal{L}^{2}(M)
\end{equation}
and put
\begin{equation}\label{eq-3-5}
  \mathcal{S}_p(M) = \big\{\, \xi \in \mathcal{L}^{2}(M) \mid \|\xi\|_{p}< \infty\,\big\}.
\end{equation}
It is not hard to check that $\|\cdot\|_{p}$ is a Hilbert norm and $\mathcal{S}_p(M)$ becomes a Hilbert space
with $\|\cdot\|_{p}$. Moreover, the inner product corresponding to $\|\cdot\|_{p}$ is given by
\begin{equation}\label{eq-3-6}
  \langle \xi,\eta\rangle_p
  = \sum_{\sigma\in \Gamma}\lambda_{\sigma}^{2p}\overline{\langle Z_{\sigma},\xi\rangle} \langle Z_{\sigma}, \eta\rangle,\quad
  \xi,\, \eta \in \mathcal{S}_p(M).
\end{equation}
Here $\overline{\langle Z_{\sigma},\xi\rangle}$ means the complex conjugate of $\langle Z_{\sigma},\xi\rangle$.

\begin{lemma}\label{lem-3-2}
For $p\geq 0$, one has $\{Z_{\sigma}\mid \sigma\in\Gamma\} \subset \mathcal{S}_p(M)$ and moreover the system
$\{\lambda^{-p}_{\sigma}Z_{\sigma}\mid \sigma\in\Gamma\}$ forms an orthonormal basis for $\mathcal{S}_p(M)$.
\end{lemma}

\begin{proof}
For $\sigma \in \Gamma$, a direct calculation gives $\|Z_{\sigma}\|_p = \lambda_{\sigma}^{p}<\infty$, which means that
$Z_{\sigma} \in \mathcal{S}_p(M)$. Clearly $\{\lambda^{-p}_{\sigma}Z_{\sigma} \mid \sigma\in\Gamma\}$ is an orthonormal system
in $\mathcal{S}_p(M)$. To complete the proof, we need only to show that it is also total in $\mathcal{S}_p(M)$.
In fact, we have
\begin{equation}\label{eq-3-7}
  \langle \lambda_{\sigma}^{-p}Z_{\sigma}, \xi \rangle_p
   = \sum_{\tau\in \Gamma}\lambda_{\tau}^{2p}\overline{\langle Z_{\tau},\lambda^{-p}_{\sigma}Z_{\sigma}\rangle} \langle Z_{\tau}, \xi\rangle
   = \lambda_{\sigma}^p\langle Z_{\sigma}, \xi\rangle,\quad \xi \in \mathcal{S}_p(M).
\end{equation}
So, if $\xi \in \mathcal{S}_p(M)$ satisfies that $\langle \lambda^{-p}_{\sigma}Z_{\sigma}, \xi \rangle_p=0$ for all $\sigma \in \Gamma$, then
it must satisfy that $\langle Z_{\sigma}, \xi\rangle =0$ for all $\sigma \in \Gamma$, which implies that $\xi=0$ because the system
$\{Z_{\sigma}\mid \sigma \in \Gamma\}$ is an orthonormal basis for$L^2(M)$.
Thus $\{\lambda^{-p}_{\sigma}Z_{\sigma}\mid \sigma\in\Gamma\}$ is total in $\mathcal{S}_p(M)$.
\end{proof}

It is easy to see that $\lambda_{\sigma}\geq 1$ for all $\sigma\in \Gamma$. This implies that $\|\cdot\|_p \leq \|\cdot\|_q$
and $\mathcal{S}_q(M)\subset \mathcal{S}_p(M)$
whenever $0\leq p \leq q$. Thus we actually get a chain of Hilbert spaces of functionals of $M$:
\begin{equation}\label{eq-3-8}
 \cdots \subset \mathcal{S}_{p+1}(M) \subset \mathcal{S}_p(M)\subset  \cdots \subset \mathcal{S}_1(M) \subset \mathcal{S}_0(M)=\mathcal{L}^{2}(M).
\end{equation}
We now put
\begin{equation}\label{eq-3-9}
  \mathcal{S}(M)=\bigcap ^{\infty}_{p=0}\mathcal{S}_{p}(M)
\end{equation}
and endow it with the topology generated by the norm sequence $\{\|\cdot \|_{p}\}_{p\geq 0}$.
Note that, for each $ p\geq 0$, $\mathcal{S}_p(M)$ is just the completion of $\mathcal{S}(M)$ with respect to $\|\cdot\|_{p}$.
Thus $\mathcal{S}(M)$ is a countably-Hilbert space (\cite{becnel, gelfand}).
The next lemma, however, shows that $\mathcal{S}(M)$ even has a much better property.

\begin{lemma}\label{lem-3-3}
The space $\mathcal{S}(M)$ is a nuclear space, namely for any $p\geq 0$,
there exists $q> p$ such that the inclusion mapping $i_{pq}\colon \mathcal{S}_q(M) \rightarrow \mathcal{S}_p(M)$
defined by $i_{pq}(\xi)=\xi$ is a Hilbert-Schmidt operator.
\end{lemma}

\begin{proof}
Let $p\geq 0$. Then there exists $q>p$ such that $2(q-p)>1$.
By Lemma~\ref{lem-3-2},
$\{\lambda^{-q}_{\sigma}Z_{\sigma} \mid \sigma\in\Gamma\}$ is an orthonormal basis for $\mathcal{S}_q(M)$.
Thus, it follows from Lemma~\ref{lem-wz} that
\begin{equation}\label{eq-3-10}
\|i_{pq}\|_{HS}^{2}=\sum_{\sigma\in\Gamma}\|i_{pq}(\lambda^{-q}_{\sigma}Z_{\sigma})\|_{p}^{2}
=\sum_{\sigma\in\Gamma}\lambda^{-2(q-p)}_{\sigma}<\infty,
\end{equation}
where $\|\cdot\|_{HS}$ denotes the Hilbert-Schmidt norm of an operator. Therefore the inclusion mapping
$i_{pq}\colon \mathcal{S}_q(M) \rightarrow \mathcal{S}_p(M)$ is a Hilbert-Schmidt operator.
\end{proof}

For $p\geq 0$, we denote by $\mathcal{S}_p^*(M)$ the dual of $\mathcal{S}_p(M)$ and $\|\cdot\|_{-p}$
the norm of $\mathcal{S}_p^*(M)$. Then $\mathcal{S}_p^*(M)\subset \mathcal{S}_q^*(M)$ and
$\|\cdot\|_{-p} \geq \|\cdot\|_{-q}$ whenever $0\leq p \leq q$.
The lemma below is then an immediate consequence of the general theory of countably-Hilbert spaces (see, e.g., \cite{becnel} or \cite{gelfand}).

\begin{lemma}\label{lem-3-4}
Let $\mathcal{S}^*(M)$ the dual of $\mathcal{S}(M)$ and endow it with the strong topology. Then
\begin{equation}\label{eq-3-11}
  \mathcal{S}^*(M)=\bigcup_{p=0}^{\infty}\mathcal{S}_p^*(M)
\end{equation}
and moreover the inductive limit topology on $\mathcal{S}^*(M)$ given by space sequence $\{\mathcal{S}_p^*(M)\}_{p\geq 0}$ coincides with the strong topology.
\end{lemma}

We mention that, by identifying $\mathcal{L}^2(M)$ with its dual, one comes to a Gel'fand triple
\begin{equation}\label{eq-3-12}
\mathcal{S}(M)\subset \mathcal{L}^2(M)\subset \mathcal{S}^*(M),
\end{equation}
which we refer to as the Gel'fand triple associated with $M$.

\begin{theorem}\label{thr-3-5}
The system $\{Z_{\sigma} \mid \sigma \in \Gamma\}$ is contained in $\mathcal{S}(M)$ and moreover it forms a basis for $\mathcal{S}(M)$ in the sense that
\begin{equation}\label{eq-3-13}
  \xi = \sum_{\sigma \in \Gamma} \langle Z_{\sigma}, \xi\rangle Z_{\sigma}, \quad \xi \in \mathcal{S}(M),
\end{equation}
where $\langle\cdot,\cdot\rangle$ is the inner product of $\mathcal{L}^2(M)$ and the series converges in the topology of $\mathcal{S}(M)$.
\end{theorem}

\begin{proof}
It follows from Lemma~\ref{lem-3-2} and the definition of $\mathcal{S}(M)$ that the system $\{Z_{\sigma}\mid \sigma\in\Gamma\}$ is contained in $\mathcal{S}(M)$.
Let $\xi \in \mathcal{S}(M)$. Then, for each $p\geq 0$, we have $\xi \in \mathcal{S}_p(M)$, which together with Lemma~\ref{lem-3-2} gives
\begin{equation*}
  \xi = \sum_{\sigma \in \Gamma} \big\langle \lambda^{-p}_{\sigma}Z_{\sigma}, \xi\big\rangle_p \lambda^{-p}_{\sigma}Z_{\sigma},
\end{equation*}
where the series on the righthand side converges in norm $\|\cdot\|_p$. On the other hand, we find
\begin{equation}\label{eq-3-14}
  \langle \lambda_{\sigma}^{-p}Z_{\sigma}, \xi \rangle_p
      = \lambda_{\sigma}^p\langle Z_{\sigma}, \xi\rangle,\quad p\geq 0.
\end{equation}
Thus
\begin{equation*}
  \xi = \sum_{\sigma \in \Gamma} \langle Z_{\sigma}, \xi\rangle Z_{\sigma}
\end{equation*}
and the series on the righthand side converges in $\|\cdot\|_p$ for each $p\geq 0$, namely in the topology of $\mathcal{S}(M)$.
\end{proof}

\begin{definition}\label{def-3-1}
Elements of $\mathcal{S}^*(M)$ are called generalized functionals of $M$, while elements of $\mathcal{S}(M)$ are called testing functionals of $M$.
\end{definition}

As mentioned above, by identifying $\mathcal{L}^2(M)$ with its dual, one has the Gel'fand triple associated with $M$, namely
$\mathcal{S}(M)\subset \mathcal{L}^2(M)\subset \mathcal{S}^*(M)$,
which justifies this definition.

\section{Characterization theorems}\label{sec-4}

Let $M=(M_n)_{n\in \mathbb{N}}$ be the same as in Section~\ref{sec-3}.
In this section, we establish some characterization theorems for generalized functionals of $M$,
which are our main results.

We continue to use the notions and notation made in previous sections. Additionally, we denote by $\langle\!\langle \cdot,\cdot\rangle\!\rangle$
the canonical bilinear form on $\mathcal{S}^*(M)\times \mathcal{S}(M)$, namely
\begin{equation}\label{eq-4-1}
  \langle\!\langle \Phi,\xi\rangle\!\rangle = \Phi(\xi),\quad \Phi\in \mathcal{S}^*(M),\, \xi\in \mathcal{S}(M).
\end{equation}
Note that $\langle\cdot,\cdot\rangle$ denotes the inner product of $\mathcal{L}^2(M)$, which is different from
$\langle\!\langle \cdot,\cdot\rangle\!\rangle$.

Recall that $\{Z_{\sigma}\mid \sigma\in\Gamma\} \subset \mathcal{S}(M)$. This allows us
to introduce the following definition.

\begin{definition}\label{def-4-1}
For $\Phi \in S^*(\Omega)$, its Fock transform is the function $\widehat{\Phi}$ on $\Gamma$ given by
\begin{equation}\label{eq-4-2}
  \widehat{\Phi}(\sigma) = \langle\!\langle \Phi, Z_{\sigma}\rangle\!\rangle,\quad \sigma \in \Gamma,
\end{equation}
where $\langle\!\langle \cdot,\cdot\rangle\!\rangle$ is the canonical bilinear form.
\end{definition}

The theorem below shows that a generzlized functional of $M$ is completely determined by its Fock tramsform.

\begin{theorem}\label{thr-4-1}
Let $\Phi$, $\Psi \in \mathcal{S}^*(M)$. Then $\Phi=\Psi$ if and only if $\widehat{\Phi}=\widehat{\Psi}$.
\end{theorem}

\begin{proof}
Clearly, we need only to prove the ``if" part. To do so, we assume $\widehat{\Phi}=\widehat{\Psi}$. Then, for each
$\xi \in \mathcal{S}(M)$, by using Theorem~\ref{thr-3-5} and the continuity of $\Phi$ and $\Psi$ we have
\begin{equation*}
\begin{split}
\langle\!\langle \Phi,\xi\rangle\!\rangle
      &= \sum_{\sigma \in \Gamma} \langle Z_{\sigma}, \xi\rangle \langle\!\langle \Phi,Z_{\sigma}\rangle\!\rangle
       = \sum_{\sigma \in \Gamma} \langle Z_{\sigma}, \xi\rangle \widehat{\Phi}(\sigma)\\
      &= \sum_{\sigma \in \Gamma} \langle Z_{\sigma}, \xi\rangle \widehat{\Psi}(\sigma)
      = \sum_{\sigma \in \Gamma} \langle Z_{\sigma}, \xi\rangle \langle\!\langle \Psi,Z_{\sigma}\rangle\!\rangle
      = \langle\!\langle \Psi,\xi\rangle\!\rangle.
\end{split}
\end{equation*}
Thus $\Phi=\Psi$.
\end{proof}

\begin{theorem}\label{thr-4-2}
Let $\Phi \in \mathcal{S}^*(M)$. Then there exist constants $C\geq 0$ and $p\geq 0$ such that
\begin{equation}\label{eq-4-3}
  |\widehat{\Phi}(\sigma)| \leq C\lambda_{\sigma}^p,\quad \sigma \in \Gamma.
\end{equation}
\end{theorem}

\begin{proof}
By Lemma~\ref{lem-3-4}, there exists some $p\geq 0$ such that $\Phi \in \mathcal{S}^*_p(M)$.
Now write $C=\|\Phi\|_{-p}$. Then, for each $\sigma \in \Gamma$, we have
\begin{equation*}
|\widehat{\Phi}(\sigma)|
   = |\langle\!\langle \Phi, Z_{\sigma}\rangle\!\rangle|
   \leq C\|Z_{\sigma}\|_p
   = C\lambda_{\sigma}^p.
\end{equation*}
This completes the proof.
\end{proof}

\begin{theorem}\label{thr-4-3}
Let $F$ be a function on $\Gamma$ satisfying
\begin{equation}\label{eq-4-4}
  |F(\sigma)| \leq C\lambda_{\sigma}^p,\quad \sigma \in \Gamma
\end{equation}
for some constants $C\geq 0$ and $p\geq 0$.
Then there exists a unique $\Phi \in \mathcal{S}^*(M)$ such that $F=\widehat{\Phi}$.
\end{theorem}

\begin{proof}
We first show that the series $\sum_{\sigma \in \Gamma} \langle Z_{\sigma}, \xi\rangle F(\sigma)$ absolutely converges for each $\xi \in \mathcal{S}(M)$,
where $\langle\cdot,\cdot\rangle$ is the inner product of $\mathcal{L}^2(M)$.
In fact, by taking $q>p+\frac{1}{2}$, we have
\begin{equation*}
  \sum_{\sigma \in \Gamma}|\langle Z_{\sigma}, \xi\rangle F(\sigma)|
   = \sum_{\sigma \in \Gamma}  |\lambda_{\sigma}^q\langle Z_{\sigma}, \xi\rangle|  \lambda_{\sigma}^{-q}|F(\sigma)|,\quad \xi \in \mathcal{S}(M),
\end{equation*}
which, together with the following equalities
\begin{equation*}
 \lambda_{\sigma}^q\langle Z_{\sigma}, \xi\rangle = \langle\lambda_{\sigma}^{-q} Z_{\sigma}, \xi\rangle_{q},\quad
\|\xi\|_q^2 = \sum_{\sigma \in \Gamma}|\langle\lambda_{\sigma}^{-q} Z_{\sigma}, \xi\rangle_{q}|^2
\end{equation*}
and the condition described by (\ref{eq-4-4}), gives
\begin{equation}\label{eq-4-5}
\begin{split}
  \sum_{\sigma \in \Gamma}|\langle Z_{\sigma}, \xi\rangle F(\sigma)|
      &\leq C\sum_{\sigma \in \Gamma}|\langle\lambda_{\sigma}^{-q} Z_{\sigma}, \xi\rangle_{q}|
          \lambda_{\sigma}^{-(q-p)}\\
   &\leq C\bigg[\sum_{\sigma \in \Gamma}|\langle\lambda_{\sigma}^{-q} Z_{\sigma}, \xi\rangle_{q}|^2\bigg]^{\frac{1}{2}}
         \bigg[\sum_{\sigma \in \Gamma}\lambda_{\sigma}^{-2(q-p)}\bigg]^{\frac{1}{2}}\\
   &= C\|\xi\|_q\bigg[\sum_{\sigma \in \Gamma}\lambda_{\sigma}^{-2(q-p)}\bigg]^{\frac{1}{2}}
 \end{split}
\end{equation}
with $\xi \in \mathcal{S}(M)$, which together with Lemma~\ref{lem-wz} yields that
\begin{equation*}
\sum_{\sigma \in \Gamma}|\langle Z_{\sigma}, \xi\rangle F(\sigma)|< \infty,\quad \xi \in \mathcal{S}(M),
\end{equation*}
namely the series $\sum_{\sigma \in \Gamma} \langle Z_{\sigma}, \xi\rangle F(\sigma)$ absolutely converges for each $\xi \in \mathcal{S}(M)$.

We now show that there exists a unique $\Phi \in \mathcal{S}^*(M)$ such that $F=\widehat{\Phi}$. In fact,
we can define a functional $\Phi$ on $\mathcal{S}(M)$ as
\begin{equation}\label{eq-4-6}
  \Phi(\xi) = \sum_{\sigma \in \Gamma} \langle Z_{\sigma}, \xi\rangle F(\sigma), \quad \xi \in \mathcal{S}(M).
\end{equation}
Clearly, $\Phi$ is a linear functional on $\mathcal{S}(M)$. Moreover, by using (\ref{eq-4-5}), we find
\begin{equation*}
  |\Phi(\xi)| \leq C\bigg[\sum_{\sigma \in \Gamma}\lambda_{\sigma}^{-2(q-p)}\bigg]^{\frac{1}{2}}\|\xi\|_{q}, \quad \xi \in \mathcal{S}(M),
\end{equation*}
where $q>p+\frac{1}{2}$. Thus $\Phi \in \mathcal{S}^*(M)$. A simple calculation gives that $\widehat{\Phi} =F$ and, finally, Theorem~\ref{thr-4-1} implies that
such a $\Phi \in \mathcal{S}^*(M)$ is unique.
\end{proof}

Theorems~\ref{thr-4-2} and \ref{thr-4-3} characterize generalized functionals of $M$ through their Fock transforms.
As an immediate consequence of these two theorems, we come to the next corollary, which offers a criterion for checking whether or not
a function on $\Gamma$ is the Fock transform of a generalized functional of $M$.

\begin{corollary}\label{corol-4-1}
Let $F$ be a function on $\Gamma$. Then $F$ is the Fock transform of an element of $\mathcal{S}^*(M)$ if and only if it satisfies
\begin{equation}\label{eq-4-7}
  |F(\sigma)| \leq C\lambda_{\sigma}^p,\quad  \sigma \in \Gamma,
\end{equation}
where $C\geq 0$ and $p\geq 0$ are some constants independant of $\sigma \in \Gamma$.
\end{corollary}

\begin{remark}
The condition described by (\ref{eq-4-7}) is actually a type of growth condition. This corollary then shows that growth condition is enough to characterize
generalized functionals of $M$.
\end{remark}

Let $\eta \in \mathcal{S}(M)$. Then there exists a continuous linear functional $\Phi_{\eta}$ on $\mathcal{L}^2(M)$ such that
$\|\eta\| = \sup\!\big\{\, |\Phi_{\eta}(\xi)| \mid \|\xi\|= 1,\, \xi\in \mathcal{L}^2(M)\, \big\}$
and
\begin{equation}\label{eq-4-8}
  \Phi_{\eta}(\xi) = \langle\eta,\xi\rangle,\quad  \xi \in \mathcal{L}^2(M),
\end{equation}
where $\langle\cdot,\cdot\rangle$ and $\|\cdot\|$ are the inner product and norm of $\mathcal{L}^2(M)$, respectively.
As a functional on $\mathcal{S}(M)$, $\Phi_{\eta}$ is obviously continuous with respect to the topology of $\mathcal{S}(M)$, thus $\Phi_{\eta} \in \mathcal{S}^*(M)$.
Based on these observations, we come to the next theorem,  which actually offers a characterization of testing functionals of $M$.

\begin{theorem}
Let $F$ be a function on $\Gamma$. If $F$ satisfies that for each $p\geq 0$ there exists $C\geq 0$ such that
\begin{equation}\label{eq-4-9}
  |F(\sigma)| \leq C\lambda_{\sigma}^{-p},\quad \sigma \in \Gamma,
\end{equation}
then there exists a unique $\eta \in \mathcal{S}(M)$ such that $\widehat{\Phi_{\eta}}=F$. Conversely, if $F=\widehat{\Phi_{\eta}}$ for some
$\eta \in \mathcal{S}(M)$, then for each $p\geq 0$ there exists $C\geq 0$ such that (\ref{eq-4-9}) holds.
\end{theorem}

\begin{proof}
The second part of the theorem can be proved easily. Here we only give a proof to the first part.

To do so, we consider the series $\sum_{\sigma \in \Gamma}\overline{F(\sigma)}Z_{\sigma}$ in $\mathcal{S}(M)$.
Let $p\geq 0$. Then we can take $q>p$ such that $2(q-p)>1$. By the condition on $F$, there exists a constant $C_0\geq 0$ such that
\begin{equation*}
  |F(\sigma)| \leq C_0\lambda_{\sigma}^{-q},\quad \sigma \in \Gamma,
\end{equation*}
which, together with Lemma~\ref{lem-wz}, yields
\begin{equation}\label{eq-4-10}
  \sum_{\sigma \in \Gamma}\|\overline{F(\sigma)}Z_{\sigma}\|_p^2
      =  \sum_{\sigma \in \Gamma}|F(\sigma)|^2 \lambda_{\sigma}^{2p}
     \leq C_0^2\sum_{\sigma \in \Gamma}\lambda_{\sigma}^{-2(q-p)}
     < \infty.
\end{equation}
On the other hand, $\sum_{\sigma \in \Gamma}\overline{F(\sigma)}Z_{\sigma}$ is an orthogonal series with respect to $\|\cdot\|_p$.
This together with (\ref{eq-4-10}) implies that it converges in $\|\cdot\|_p$.
Thus, by the arbitrariness of the choice of $p\geq 0$, it converges actually in $\mathcal{S}(M)$.

Now we write $\eta = \sum_{\sigma \in \Gamma}\overline{F(\sigma)}Z_{\sigma}$. A simple calculation gives that
$F(\sigma) = \langle \eta, Z_{\sigma} \rangle$ for $\sigma \in \Gamma$, which together with (\ref{eq-4-8}) leads to $\widehat{\Phi_{\eta}}=F$.
Clearly, such an $\eta \in \mathcal{S}(M)$ is unique.
\end{proof}

\section{Applications }\label{sec-5}

In the last section, we show some applications of our results obtained in previous sections.

Let $M=(M_n)_{n\in \mathbb{N}}$ be the same as in Section~\ref{sec-3}. We continue to use the notions and notation made in previous sections.
Recall that elements of $\mathcal{S}^*(M)$ are called generalized functionals of $M$.
The theorem below actually gives norm estimates to generalized functionals of $M$.

\begin{theorem}\label{thr-5-1}
Let $\Phi$ be a generalized functional of $M$ and $C\geq 0$, $p\geq 0$ two constants such that
\begin{equation}\label{eq-5-1}
  |\widehat{\Phi}(\sigma)| \leq C\lambda_{\sigma}^p,\quad \sigma \in \Gamma.
\end{equation}
Then for $q> p+\frac{1}{2}$ one has
\begin{equation}\label{eq-5-2}
  \|\Phi\|_{-q} \leq C\bigg[\sum_{\sigma \in \Gamma}\lambda_{\sigma}^{-2(q-p)}\bigg]^{\frac{1}{2}},
\end{equation}
in particular $\Phi \in \mathcal{S}_q^*(M)$.
\end{theorem}

\begin{proof}
Let $\xi \in \mathcal{S}(M)$. Then, by Theorem~\ref{thr-3-5} and the continuity of $\Phi$, we have
\begin{equation*}
\Phi(\xi)
    = \sum_{\sigma \in \Gamma} \langle Z_{\sigma}, \xi\rangle \Phi(Z_{\sigma})
    = \sum_{\sigma \in \Gamma} \langle Z_{\sigma}, \xi\rangle \widehat{\Phi}(\sigma).
\end{equation*}
On the other hand, by using (\ref{eq-5-1}) and (\ref{eq-3-14}), we get
\begin{equation*}
  |\langle Z_{\sigma}, \xi\rangle \widehat{\Phi}(\sigma)|
    \leq C  \lambda_{\sigma}^p |\langle Z_{\sigma}, \xi\rangle|
   = C\lambda_{\sigma}^{-(q-p)}|\lambda_{\sigma}^q\langle Z_{\sigma}, \xi\rangle|
   = C\lambda_{\sigma}^{-(q-p)}|\langle \lambda_{\sigma}^{-q} Z_{\sigma}, \xi\rangle_{q}|,
\end{equation*}
where $\sigma$ runs over $\Gamma$. Thus
\begin{equation*}
\begin{split}
  |\Phi(\xi)|
     &\leq \sum_{\sigma \in \Gamma} |\langle Z_{\sigma}, \xi\rangle \widehat{\Phi}(\sigma)|\\
     &\leq C\sum_{\sigma \in \Gamma} \lambda_{\sigma}^{-(q-p)}|\langle \lambda_{\sigma}^{-q} Z_{\sigma}, \xi\rangle_{q}|\\
     & \leq C\bigg[\sum_{\sigma \in \Gamma} \lambda_{\sigma}^{-2(q-p)}\bigg]^{\frac{1}{2}}
            \bigg[\sum_{\sigma \in \Gamma} |\langle \lambda_{\sigma}^{-q} Z_{\sigma}, \xi\rangle_{q}|^2\bigg]^{\frac{1}{2}}\\
     &= C\bigg[\sum_{\sigma \in \Gamma} \lambda_{\sigma}^{-2(q-p)}\bigg]^{\frac{1}{2}} \|\xi\|_q,
\end{split}
\end{equation*}
which implies (\ref{eq-5-2}).
\end{proof}

\begin{example}\label{example-5-1}
Consider the counting measure $\#(\cdot)$ over $\mathbb{N}$. It can be shown that, as a function on $\Gamma$,  $\#(\cdot)$  satisfies
\begin{equation*}
  |\#(\sigma)| \leq \lambda_{\sigma},\quad \sigma \in \Gamma.
\end{equation*}
Thus, by Corollary~\ref{corol-4-1} and Theorem~\ref{thr-5-1}, $\#(\cdot)$ is the Fock transform of a certain generalized functional $\Phi_{\#}$ of $M$, and moreover
$\Phi_{\#}$ has a norm estimate like
\begin{equation}\label{ }
\|\Phi_{\#}\|_{-q} \leq  \bigg[\sum_{\sigma \in \Gamma}\lambda_{\sigma}^{-2(q-1)}\bigg]^{\frac{1}{2}}
 \leq \exp \bigg[\frac{1}{2}\sum_{k=1}^{\infty}k^{-2(q-1)}\bigg],
\end{equation}
where $q> \frac{3}{2}$.
\end{example}

\begin{example}\label{example-5-2}
Consider the function $F(\sigma) = \sqrt{\lambda_{\sigma}}$ on $\Gamma$. Clearly, it satisfies condition~(\ref{eq-4-7})
with $C=1$ and $p=\frac{1}{2}$.
Thus, there exists a generalized functional of $M$, written as $\Phi_{\sqrt{\lambda}}$, such that
\begin{equation*}
\widehat{\Phi_{\sqrt{\lambda}}}(\sigma) = \sqrt{\lambda_{\sigma}},\quad \sigma \in \Gamma.
\end{equation*}
Moreover,  by Theorem~\ref{thr-5-1}, $\Phi_{\sqrt{\lambda}}$ has a norm estimate as below
\begin{equation}\label{ }
\|\Phi_{\sqrt{\lambda}}\|_{-q} \leq  \bigg[\sum_{\sigma \in \Gamma}\lambda_{\sigma}^{1-2q}\bigg]^{\frac{1}{2}}
 \leq \exp \bigg[\frac{1}{2}\sum_{k=1}^{\infty}k^{1-2q}\bigg],
\end{equation}
where $q> 1$.
\end{example}

\begin{theorem}\label{thr-5-2}
Let $\Phi \in \mathcal{S}_p^*(M)$, where $p\geq 0$.  Then the norm of $\Phi$ in $\mathcal{S}_p^*(M)$ satisfies
\begin{equation}\label{ }
  \|\Phi\|_{-p}^2 = \sum_{\sigma \in \Gamma} \lambda_{\sigma}^{-2p} |\widehat{\Phi}(\sigma)|^2.
\end{equation}
\end{theorem}

\begin{proof}
By the Riesz representation theorem \cite{{muscat}}, there exists a unique $\eta \in \mathcal{S}_p(M)$ such that
$\|\eta\|_p = \|\Phi\|_{-p}$
and
\begin{equation}\label{ }
  \Phi(\xi) = \langle\eta,\xi\rangle_{p},\quad  \xi \in \mathcal{S}_p(M),
\end{equation}
which, together with Lemma~\ref{lem-3-2}, gives
\begin{equation*}
\|\Phi\|_{-p}^2
     =\|\eta\|_p^2
     = \sum_{\sigma \in \Gamma} |\langle\lambda_{\sigma}^{-p}Z_{\sigma}, \eta \rangle_{p}|^2
     = \sum_{\sigma \in \Gamma} \lambda_{\sigma}^{-2p}|\widehat{\Phi}(\sigma)|^2.
\end{equation*}
This completes the proof.
\end{proof}

In general, the usual product of two generalized functionals of $M$ is no longer a generalized functional of $M$. This means that the usual product
is not a multiplication in $\mathcal{S}^*(M)$. The following two examples, however, show that by using our characterization theorems one can define
other types of multiplication in $\mathcal{S}^*(M)$.

\begin{example}\label{example-5-3}
Let $\Phi$, $\Psi$ be generalized functionals of $M$. Then, by Theorem~\ref{thr-4-2}, the function $F(\sigma) = \widehat{\Phi}(\sigma)\widehat{\Psi}(\sigma)$
satisfies condition~(\ref{eq-4-7}). Thus there exists a unique generalized functional of $M$, written as $\Phi\ast\Psi$,
such that
\begin{equation*}
  \widehat{\Phi\ast\Psi}(\sigma) = \widehat{\Phi}(\sigma) \widehat{\Psi}(\sigma),\quad \sigma \in \Gamma.
\end{equation*}
We call $\Phi\ast\Psi$ the convolution of $\Phi$ and $\Psi$. It can be shown that, with $\ast$ as the multiplication, $\mathcal{S}^*(M)$
becomes an algebra.
\end{example}

\begin{remark}
In \cite{han}, the authors defined the convolution for square integrable functionals of $M$.
Here our definition of convolution actually extends that in \cite{han}.
\end{remark}

\begin{example}\label{example-5-4}
Let $\Phi$, $\Psi$ be generalized functionals of $M$. Then there exists a unique generalized functional of $M$, written as $\Phi\diamond\Psi$,
such that
\begin{equation}\label{ }
 \widehat{\Phi\diamond\Psi}(\sigma) = \sum_{\tau \subset \sigma} \widehat{\Phi}(\tau)\widehat{\Psi}(\sigma\setminus\tau),\quad \sigma \in \Gamma,
\end{equation}
where $\sum_{\tau \subset \sigma}$ means that the sum is taken over all subsets of $\sigma$. We call $\Phi\diamond\Psi$ the Wick product of $\Phi$ and $\Psi$.
\end{example}

\begin{proof}
In fact, by Lemma~\ref{lem-3-4}, there exists $p\geq 0$ such that $\Phi$, $\Psi\in \mathcal{S}_p^*(M)$. Now define
\begin{equation*}
 F(\sigma) = \sum_{\tau \subset \sigma} \widehat{\Phi}(\tau)\widehat{\Psi}(\sigma\setminus\tau),\quad \sigma \in \Gamma
\end{equation*}
and take $q> p +\frac{1}{2}$. Then, by using Theorem~\ref{thr-5-2}, we have
\begin{equation*}
\begin{split}
\sum_{\sigma\in\Gamma}\lambda_{\sigma}^{-2q}|F(\sigma)|^2
&=\sum_{\sigma\in\Gamma}\lambda_{\sigma}^{-2(q-p)}\Big|\sum_{\tau\subset\sigma}\lambda_{\tau}^{-p}\widehat{\Phi}(\tau)
  \lambda_{\sigma\setminus\tau}^{-p}\widehat{\Psi}(\sigma\setminus \tau)\Big|^2\\
&\leq \sum_{\sigma\in\Gamma}\lambda_{\sigma}^{-2(q-p)}\sum_{\tau\subset\sigma}\lambda_{\tau}^{-2p}|\widehat{\Phi}(\tau)|^2
      \sum_{\tau\subset\sigma}\lambda_{\tau}^{-2p}|\widehat{\Psi} (\tau)|^{2}\\
&\leq \sum_{\sigma\in\Gamma}\lambda_{\sigma}^{-2(q-p)}\sum_{\tau\in \Gamma}\lambda_{\tau}^{-2p}|\widehat{\Phi}(\tau)|^2
      \sum_{\tau\in \Gamma}\lambda_{\tau}^{-2p}|\widehat{\Psi} (\tau)|^{2}\\
&=\|\Phi\|_{-p}^2\|\Psi\|_{-p}^2\sum_{\sigma\in\Gamma}\lambda_{\sigma}^{-2(q-p)},
\end{split}
\end{equation*}
which implies that $|F(\sigma)| \leq C\lambda_{\sigma}^q$, $\sigma\in \Gamma$,
where
\begin{equation*}
   C=\|\Phi\|_{-p} \|\Psi\|_{-p} \bigg[\sum_{\sigma\in\Gamma}\lambda_{\sigma}^{-2(q-p)}\bigg]^{\frac{1}{2}}<\infty.
\end{equation*}
Thus, by Corollary~\ref{corol-4-1}, there exists a unique generalized functional of $M$, written as $\Phi\diamond\Psi$, satisfying
$\widehat{\Phi\diamond\Psi} =F$.
\end{proof}

\begin{remark}
In \cite{wang-z}, by using the Guichardet representation, the authors defined the Wick product for generalized functionals of Bernoulli noise.
\end{remark}

\section*{Acknowledgement}

This work is supported by National Natural Science Foundation of China (Grant No. 11461061).


\begin{thebibliography}{99}

\bibitem{albe} S. Albeverio, Yu.L. Daletsky, Yu.G. Kondratiev and L. Streit, Non-Gaussian infinite dimensional analysis, \emph{J. Funct. Anal.} 138 (1996), 311--350.

\bibitem{barhoumi} A. Barhoumi, H. Ouerdiane and A. Riahi, Pascal white noise calculus, \emph{Stochastics} 81 (2009), 323--343.

\bibitem{becnel} Jeremy J. Becnel, Equivalence of topologies and Borel fields for countably-Hilbert spaces,
\emph{Proc. Amer. Math. Soc.} 134 (2006), 581--590.

\bibitem{di} G. Di Nunno, B. Oksendal and F. Proske, White noise analysis for L\'{e}vy processes, \emph{J. Funct. Anal.} 206 (2004), 109--148.

\bibitem{emery} M. \'{E}mery, A discrete approach to the chaotic representation property, in:
\emph{S\'{e}minaire de Probabilit\'{e}s, XXXV}, Lecture Notes in Mathematics 1755, 123--138, Springer, Berlin 2001.

\bibitem{gelfand} I.M. Gel'fand and G.E. Shilov, \emph{Spaces of Fundamental and Generalized Functions, Generalized
Functions, vol. 2}, Academic Press, New York 1968.

\bibitem{han} Q. Han, C.S. Wang and Y.L. Zhou, Convolution of functionals of discrete-time normal martingales,
 \emph{Bull. Aust. Math. Soc.} 86 (2012), 224--231.

\bibitem{hida} T. Hida, H.-H. Kuo, J. Potthoff and L. Streit, \emph{White Noise: An infinite Dimensional Calculus}, Kluwer Academic, Dordrecht 1993.

\bibitem{hol} H. Holden, B. Oksendal, J. Uboe and T. Zhang, \emph{Stochastic Partial Differential Equations, A modeling, white noise functional approach},
Birkhauser, Boston 1996.

\bibitem{hu} Y. Hu and B. Oksendal, Fractional white noise calculus and applications to finance, \emph{Infin. Dimens. Anal. Quantum Probab. Relat. Top.} 6 (2003), 1--32.

\bibitem{huang} Z.Y. Huang and J.A. Yan, \emph{Introduction to Infinite Dimensional Stochastic Analysis},  Kluwer Academic, Dordrecht 1999.

\bibitem{ito} Y. Ito, Generalized Poisson functionals, \emph{Probab. Theory and Related Fields} 77 (1988), 1--28.

\bibitem{kuo} H.H. Kuo, \emph{White Noise Distribution Theory}, CRC, Boca Raton 1996.

\bibitem{lee} Y.-J. Lee and H.-H. Shih, The Segal-Bargmann transform for L\'{e}vy functionals, \emph{J. Funct. Anal.} 168 (1999), 46--83.

\bibitem{mot} R. Motwani and P. Raghavan, \emph{Randomized Algorithms}, Cambridge University Press, New York 1995.

\bibitem{muscat} J. Muscat, \emph{Functional Analysis: An Introduction to Metric Spaces, Hilbert Spaces, and Banach Algebras},
Springer International Publishing, Switzerland 2014.

\bibitem{obata} N. Obata, \emph{White Noise Calculus and Fock Space}, Springer-Verlag, Berlin 1994.

\bibitem{pott} J. Potthoff and  L. Streit, A characterization of Hida distributions, \emph{J. Funct. Anal.} 101 (1991), 212--229.

\bibitem{privault} N. Privault, Stochastic analysis of Bernoulli processes, \emph{Probab. Surv.} 5 (2008), 435--483.

\bibitem{rud} J. Rudnick and G. Gaspari, \emph{Elements of the Random Walk}, Cambridge University Press, Cambridge 2004.

\bibitem{wang-lc} C.S. Wang, Y.C. Lu and H.F. Chai, An alternative approach to Privault's discrete-time chaotic
calculus, \emph{J. Math. Anal. Appl.} 373 (2011), 643--654.

\bibitem{wang-z} C.S. Wang and J.H. Zhang, Wick analysis for Bernoulli noise functionals, \emph{Journal of Function Spaces}, Volume 2014, Article ID 727341, 2014.

\end{thebibliography}
\end{document}